\documentclass[12pt]{amsart}
\usepackage{amsfonts,amssymb,amscd,amstext,mathrsfs,color}

\input xy
\xyoption{all}
 
\newtheorem{theorem}{Theorem}
\newtheorem{corollary}{Corollary}

\newtheorem{proposition}{Proposition}
\newtheorem*{question}{Question}
\theoremstyle{definition}
\newtheorem{remark}{Remark}

\newcommand{\C}{\mathbb{C}}
\newcommand{\Aut}{{\operatorname{Aut}}}
\newcommand{\id}{{\operatorname{id}}}
\newcommand{\rne}{{\operatorname{rne}}}
\newcommand{\tam}{{\operatorname{tam}}}

\begin{document}

\title{Generic conservative dynamics on Stein manifolds \\ with the volume density property}

\author{Leandro Arosio and Finnur L\'arusson}

\address{Dipartimento Di Matematica, Universit\`a di Roma \lq\lq Tor Vergata\rq\rq, Via Della Ricerca Scientifica 1, 00133 Roma, Italy}
\email{arosio@mat.uniroma2.it}

\address{Discipline of Mathematical Sciences, University of Adelaide, Adelaide SA 5005, Australia}
\email{finnur.larusson@adelaide.edu.au}


\thanks{L.~Arosio was partially supported  by  PRIN {\sl  Real and Complex Manifolds: Geometry and Holomorphic Dynamics} n.~2022AP8HZ9, by INdAM, and by the   MUR Excellence Department Project MatMod@TOV
CUP:E83C23000330006}

\date{29 July 2025}


\begin{abstract}  We study the dynamics of generic volume-preserving automorphisms $f$ of a Stein manifold $X$ of dimension at least $2$ with the volume density property.  Among such $X$ are all connected linear algebraic groups (except $\C$ and $\C^*$) with a left- or right-invariant Haar form.  We show that a generic $f$ is chaotic and of infinite topological entropy, and that the transverse homoclinic points of each of its saddle periodic points are dense in $X$.  We present analogous results with similar proofs in the non-conservative case.  We also prove the Kupka--Smale theorem in the conservative setting.
\end{abstract}

\maketitle

\tableofcontents

\section{Introduction and main results} 
\label{sec:intro}

\noindent
In this paper, we continue our investigations of holomorphic dynamics on highly flexible complex manifolds \cite{AL2019, AL2020, AL2022, AL2025}, built on the groundbreaking work of Forn\ae ss and Sibony \cite{FS1997}.  We focus on the dynamics of generic volume-preserving holomorphic automorphisms of a Stein manifold $X$ of dimension at least $2$ with the volume density property with respect to a holomorphic volume form on $X$.  Among such manifolds are all connected linear algebraic groups (except $\C$ and $\C^*$) with a left- or right-invariant Haar form, such as the prototypical example $\C^n$ with the form $dz_1\wedge\cdots\wedge dz_n$ and $(\C^*)^n$ with the form $(z_1\cdots z_n)^{-1}dz_1\wedge\cdots\wedge dz_n$, $n\geq 2$.  

A complex manifold $X$ with a holomorphic volume form $\omega$ has the volume density property if, in the Lie algebra of holomorphic vector fields on $X$ that are divergence-free with respect to $\omega$, the complete vector fields generate a dense Lie subalgebra.  If $X$ is Stein and $\dim X\geq 2$, this implies that the volume-preserving automorphisms of $X$ exhibit flexible behaviour that is studied in Anders\'en--Lempert theory (see \cite[Chapter 4]{Forstneric2017} and \cite{FK2022}).  Such manifolds are thus in a sense opposite to Kobayashi hyperbolic manifolds. Given that Kobayashi hyperbolicity prevents any form of chaotic dynamics, it is natural to ask to what extent chaotic dynamics is possible on $X$.  Recall that an endomorphism is \emph{chaotic} in the sense of Devaney \cite{Devaney1989} if its periodic points are dense and it has a dense forward orbit.  The main question studied in our previous work \cite{AL2019, AL2020} is the following.

\begin{question}  Let $X$ be a Stein manifold of dimension at least $2$ with the volume density property with respect to a holomorphic volume form $\omega$.  Are generic volume-preserving automorphisms of $X$ chaotic?
\end{question}

Genericity is with respect to the compact-open topology, separable and defined by a complete metric, on the group $\Aut_\omega X$ of automorphisms of $X$ that preserve $\omega$. 

It follows from the work of Forn\ae ss and Sibony \cite{FS1997} that the answer to the question is affirmative if $X$ is $\C^n$, $n\geq 2$, with the standard volume form $dz_1\wedge\cdots\wedge dz_n$.  In \cite{AL2019, AL2020}, we gave affirmative answers under cohomological assumptions on the volume form.  In this paper, we give a conclusive affirmative answer to the question.

\begin{theorem}\label{main1}
A generic volume-preserving automorphism of a Stein manifold of dimension at least $2$ with the volume density property is chaotic.
\end{theorem}

Following Forn\ae ss and Sibony's strategy, we prove this theorem as a corollary of the following result.  As in \cite{AL2019}, we call an automorphism \emph{expelling} if the $G_\delta$ set of points with unbounded forward orbit is dense.

\begin{proposition}\label{p:expelling}
A generic volume-preserving automorphism of a Stein manifold of dimension at least $2$ with the volume density property is expelling.
\end{proposition} 

The proof of Proposition \ref{p:expelling} combines the argument of Forn\ae ss and Sibony, as generalised in \cite{AL2019,AL2020}, with a theorem of Barrett, Bedford, and Dadok \cite{BBD1989} on torus actions. 

Now let $f$ be a generic element of $\Aut_\omega X$.  It follows from Proposition \ref{p:expelling} (as in \cite[Theorem 3]{AL2019}) that  $f$ has a saddle fixed point with a dense stable manifold.  On the other hand, arguing as in the proof of \cite[Theorem 8]{AL2025}, one can show that $f$ has a saddle fixed point with a transverse homoclinic point.  Combining the ideas in the proofs of these two results, we obtain the following stronger theorem.

\begin{theorem}\label{main2}
Let $f$ be a generic volume-preserving automorphism of a Stein manifold $X$ of dimension at least $2$ with the volume density property.  The transverse homoclinic points of each saddle periodic point of $f$ are dense in $X$.
\end{theorem}

We move on to an investigation of the topological entropy of a generic volume-preserving automorphism of $X$.  The notion of a H\'enon-like map was introduced by Dujardin in \cite{Dujardin2004}, where the reader may also find background information on entropy.  Using Anders\'en--Lempert theory, by a small perturbation of any given volume-preserving automorphism, we can create a H\'enon-like map of arbitrarily large degree in a small ball outside a large compact set in $X$.  In this way we obtain the following result.

\begin{proposition}\label{main3}
A generic volume-preserving automorphism of a Stein manifold of dimension at least $2$ with the volume density property has infinite topological entropy.
\end{proposition}

Next, building on work of Buzzard \cite{Buzzard1998}, we complete our description of the generic conservative dynamics on a Stein manifold with the volume density property by adapting the Kupka--Smale theorem to this setting.

\begin{theorem}\label{main4}
For a generic volume-preserving automorphism of a Stein manifold of dimension at least $2$ with the volume density property, every periodic point is hyperbolic (and thus a saddle) and every homoclinic or heteroclinic point is transverse.  
\end{theorem}

As far as we know, Theorems \ref{main2} and \ref{main4} and Proposition \ref{main3} are new even for $\C^n$, $n\geq 2$, with the standard volume form.

In the final section of the paper, we outline analogues with similar proofs of some of the above results in the non-conservative case.

\section{Proof of  Proposition \ref{p:expelling} and Theorem \ref{main1}} 
\label{sec:proof1}

\begin{proof}[Proof of Proposition \ref{p:expelling}]
Let $X$ be a Stein manifold, $n=\dim X\geq 2$, with the volume density property with respect to a holomorphic volume form $\omega$.  We wish to prove our theorem \cite[Theorem 5]{AL2019} without any cohomological assumption.  The proof of the theorem shows:  Suppose that it is not the case that a generic volume-preserving automorphism of $X$ is expelling.  Then there is $f\in\Aut_\omega X$ such that the following holds.  There is a completely $f$-invariant relatively compact nonempty open set $\Omega$ in the robustly non-expelling set $\rne(f)$ of $f$ (not necessarily all of $\rne(f)$) with finitely many connected components, such that the closed subgroup $G$ of $\Aut\,\Omega$ generated by $f$ is a compact abelian Lie group (hence the product of a torus and a finite group), the orbits of $G$ in $\Omega$ are totally real, and among them is an orbit of maximal real dimension $n$.  Moreover, there are no $f$-periodic points in $\Omega$.  As explained in \cite[Footnote 3]{AL2020}, $\Omega$ may be taken to be Stein.

Choose a connected component $\Omega_0$ of $\Omega$.  It is Stein and invariant under the action of the identity component $G_0$ of $G$.  Being a subgroup of $\Aut\,\Omega$, $G$ acts effectively on $\Omega$.  Each $g\in G_0$ is the limit of a sequence of iterates of $f$.  The restriction of $g$ to any connected component of $\Omega$ may be conjugated by an iterate of $f$ to its restriction to $\Omega_0$.  Hence, if $g$ is the identity on $\Omega_0$, then $g$ is the identity on all of $\Omega$.  Thus, $G_0$ acts effectively on $\Omega_0$.  Smoothness of the action is guaranteed by the Bochner--Montgomery theorem.

The torus $G_0$ is abelian and acts effectively on $\Omega_0$, so the stabiliser of a generic point in $\Omega_0$ is trivial (see \cite[Theorem IV.3.1]{Bredon1972}).  Since there is an $n$-dimensional orbit, $\dim G_0=n$.  By \cite[Theorem 1]{BBD1989}, $\Omega_0$ with the $G_0$-action is equivariantly biholomorphic to a Reinhardt domain in $\C^n$ with the standard torus action, possibly modified by an algebraic automorphism of the torus.  But a Stein Reinhardt domain has a fixed point, which in the present context would give an $f$-periodic point in $\Omega_0$ since the iterates of $f$ are dense in $G_0$.
\end{proof}

\begin{remark}
Since the intersection of two residual subsets is residual, it immediately follows that for a generic automorphism, the set of points which have unbounded forward and backward orbits is residual.
\end{remark}

\begin{proof}[Proof of Theorem \ref{main1}]
We verify Touhey's characterisation of chaos \cite{Touhey1997}:  for every pair of nonempty open subsets $U$ and $V$ of $X$, there is a cycle that visits both of them.  To prove this, let $(U_n)_{n\in \mathbb{N}}$ be a countable basis of nonempty open sets of $X$, and for all $n, m\in \mathbb{N}$, define $S(n,m)\subset \Aut_\omega X$  as the set of automorphisms that admit a saddle cycle intersecting $U_n$ and $U_m$.  Since saddles are stable under perturbation, the sets $S(n,m)$ are open.  Using Proposition \ref{p:expelling} and Anders\'en--Lempert theory, we show that the sets $S(n,m)$ are dense; then the proof is complete by Baire.  Since the density proof is almost identical to the argument in \cite[Section 5]{AL2019}, we only sketch it.  By Proposition \ref{p:expelling}, for generic $f\in\Aut_\omega X$, there are points $x\in U_n$ and $y\in U_m$ whose forward and backward orbits both leave $f(K)$, where $K$ is a given large compact holomorphically convex set.  We use Anders\'en--Lempert theory to construct $g\in\Aut_\omega X$ close to the identity on $f(K)$, connecting chosen points of the orbits of $x$ and $y$ outside $f(K)$.  Then $g\circ f$ is close to $f$ on $K$ and has a cycle visiting both $U_n$ and $U_m$.  It is easy to avoid eigenvalues of absolute value $1$ in the relevant derivatives, so the cycle will be a saddle cycle.
\end{proof}

\begin{corollary}  
For a  generic volume-preserving automorphism of a Stein manifold $X$ of dimension at least $2$ with the volume density property the following hold.
\begin{enumerate}
\item  All points in $X$ are non-wandering.
\item  $X$ is a single chain-recurrence class.
\item  Saddle periodic points of $f$ are dense in $X$.
\end{enumerate}
\end{corollary}

\begin{proof}
The generic volume-preserving automorphism $f$ being chaotic means, first, that periodic points of $f$ are dense in $X$, which implies that all points in $X$ are non-wandering, and second, that $f$ has a dense forward orbit, which implies that every two points lie in the same chain-recurrence class.  Density of saddle periodic points was shown in the proof of Theorem \ref{main1}.
\end{proof}

\section{Proof of Theorem \ref{main2}} 
\label{sec:proof2}

\noindent
Our key new technical result is the following proposition.

\begin{proposition}\label{t:obtain-homoclinic}
Let $X$ be a Stein manifold of dimension at least 2 with the volume density property with respect to a holomorphic volume form $\omega$.  Let $p$ be a saddle periodic point of $f\in\Aut_\omega X$.  Let $U$ be a nonempty open subset of $X$.  Then every neighbourhood of $f$ in $\Aut_\omega X$ contains an automorphism with $p$ as a saddle periodic point with a transverse homoclinic point in $U$.
\end{proposition}

\begin{proof}
Let $V$ be a neighbourhood of $f$ in $\Aut_\omega X$.  Find a holomorphically convex compact subset $H$ of $X$ such that if $g\in\Aut_\omega X$ is close enough to $f$ on $H$, then $g\in V$.  We may assume that $H$ is so large that its interior contains the orbit of $p$ and, by \cite[Lemma 1]{AL2025}, that if $g$ is close enough to $f$ on $H$, then $g$ has a unique saddle periodic point $\eta(g)$ near $p=\eta(f)$, whose orbit is contained in the interior of $H$.  

By Proposition \ref{p:expelling}, there are $x\in U$ and $g$ as above, such that neither orbit $({g}^n(x))_{n\geq 0}$ nor $({g}^{-n}(x))_{n\geq  0}$ is contained in $H$.  Let $m_s\geq 0$ be the smallest integer such that ${g}^{m_s}(x)\in X\setminus H$, and let $m_u\geq 0$ be the smallest integer such that ${g}^{-m_u}(x)\in X\setminus H$.  The stable manifold $W^s_{g}(\eta(g))$ of $\eta(g)$ is not contained in $H$, so there is $\tilde y_s\in W^s_{g}(\eta(g))\setminus H$.  Let $y_s$ be the point in the forward $g$-orbit of $\tilde y_s$ that lies in $X\setminus H$ with $g^n(y_s)\in H$ for all $n\geq 1$.  Note that $y_s$  belongs to the stable manifold of a point $p_s$ in the orbit of $\eta(g)$.  Analogously, there is $\tilde y_u\in W^u_{g}(\eta(g))\setminus H$, and we let $y_u$ be the point in the backward $g$-orbit of $\tilde y_s$ that lies in $X\setminus H$ with $g^{-n}(y_u)\in H$ for all $n\geq 1$.  Let $p_u$ be the point in the orbit of $\eta(g)$ whose unstable manifold contains $y_u$.

Let $W$ be a Runge neighbourhood of $H$ not containing any of the points
$g^{m_s}(x)$, $g^{-m_u}(x)$, $y_s$, $y_u$.  We may assume that the local stable and unstable manifolds $\Gamma^s_{g}(\eta({g}),r)$ and  $\Gamma^u_{g}(\eta({g}),r)$ are contained in a coordinate ball centred at $\eta({g})$, with $r>0$ given by \cite[Lemma 2]{AL2025}.  By enlarging $H$ if necessary, we may assume that $H$ contains the polydisc $\overline{\Delta^n(0,r)}$ in its interior.  Find $n_s\geq 1$ such that ${g}^{n_s}(y_s)\in \Gamma^s_{g}(\eta({g}),r)$ and $n_u\geq 1$ such that $g^{-n_u}(y_u)\in \Gamma^u_{g}(\eta({g}),r)$.  

Let $V_s$ and $V_u$ be coordinate balls centred at ${g}^{m_s}(x)$ and $y_u$, respectively, and let 
 \[ \varphi : [0,1]\times V_s\to X,\quad \psi : [0,1]\times V_u\to X \]
 be $C^1$ isotopies such that for all $t\in [0,1]$,
\begin{itemize}
\item $\varphi_t : V_s\to X$ and $\psi_t : V_u\to X$ are holomorphic, injective, and volume-preserving,
\item $\varphi_t(V_s)$, $\psi_t(V_u)$, and $W$ are mutually disjoint, 
\item $W\cup \varphi_t(V_s)\cup \psi_t(V_u)$ is Runge,
\item $\varphi_0$ is the inclusion of $V_s$ into $X$ and  $\psi_0$ is the inclusion of $V_u$ into $X$,
\item $\varphi_1(g^{m_s}(x))=y_s$ and $\psi_1(y_u)=g^{-m_u}(x)$,
\item $d_{y_u}  \varphi_1\circ g^{m_s+m_u} \circ \psi_1(T_{y_u}W_g^u(p_u))$ is transverse to $T_{y_s}W_{g}^s(p_s)$.
\end{itemize}
By the conservative Anders\'en--Lempert theorem \cite{FR1993, KK2011}, there is a sequence $(\Phi_j)$ in $\Aut_\omega X$ such that $\Phi_j\to \id$ on $W$, $\Phi_j\to \varphi_1$ on $V_s$, and $\Phi_j\to \psi_1$ on $V_u$, uniformly on compact subsets.  Let $ g_j= g\circ \Phi_j$.  (We note that the cohomological hypothesis in the conservative Anders\'en--Lempert theorem is satisfied on $V_s$ and $V_u$ because they are balls, but is irrelevant on $W$ because we seek to approximate the identity there.)

Let $\gamma:D\to \Gamma^u_g(\eta(g),r)$ be a holomorphic graph parametrisation of the local unstable manifold  $\Gamma^u_g(\eta(g),r)$ near $g^{-n_u}(y_u)$, defined on a small polydisc $D$. If $D$ is small enough, then $(g^{n_u}\circ \gamma)(D)$ is a holomorphically embedded piece of the unstable manifold $W^u_g(p_u)$ containing $y_u$ and contained in the neighbourhood $V_u$, and
\[ (\varphi_1\circ g^{m_s+m_u}\circ\psi_1\circ g^{n_u}\circ \gamma)(D)\] 
is an embedded complex submanifold which intersects the stable manifold $W^s_g(p_s)$ transversally at $y_s$.  Also,
 \[ (g^{n_s}\circ\varphi_1\circ g^{m_s+m_u}\circ\psi_1\circ g^{n_u}\circ \gamma)(D) \]
is an embedded complex submanifold intersecting the local stable manifold $\Gamma^s_g(\eta(g),r)$ transversally at $g^{n_s}(y_s)$.

By \cite[Lemma 2]{AL2025}, there is a  holomorphic graph parametrisation $\gamma_j:D\to \Gamma^u_{ g_j}(\eta(g_j),r)$ that converges uniformly to $\gamma$ as $j\to\infty$.  It follows that the holomorphic map $ g_j^{n_s+m_s+m_u+n_u}\circ \gamma_j$ converges uniformly to $g^{n_s}\circ\varphi_1\circ g^{m_u+m_s}\circ\psi_1\circ g^{n_u}\circ \gamma$, so if $j$ is large enough, the embedded complex submanifold $ (g_j^{n_s+m_s+m_u+n_u}\circ \gamma_j)(D)$ intersects the local stable manifold $\Gamma^s_{g_j}(\eta(g_j),r)$ transversally in a homoclinic point $q$  for $g_j$.  The point $g_j^{-n_s-m_s}(q)$ is also a homoclinic point for $g_j$ and is contained in $U$.  

This proves that every neighbourhood of $f$ in $\Aut_\omega X$ contains an automorphism with a saddle periodic point in any given neighbourhood of $p$ with a transverse homoclinic point in $U$.  Conjugation by a small perturbation of the identity completes the proof of Proposition \ref{t:obtain-homoclinic}.
\end{proof}

\begin{remark}
The proof of the theorem is easily adapted to show that if $p$ and $q$ are saddle periodic points of $f$, then every neighbourhood of $f$ contains an automorphism $g$ with $p$ and $q$ as saddle periodic points such that $U$ contains a transverse heteroclinic point in $W_g^s(p)\cap W_g^u(q)$.
\end{remark}

The following proof is modelled on Xia's proof of \cite[Theorem 2]{Xia1996}.

\begin{proof}[Proof of Theorem \ref{main2}]
Exhaust $X$ by compact sets $K_1\subset K_2\subset\cdots$ and let $\mathscr H_m$ be the set of $f\in\Aut_\omega X$ such that every periodic point of $f$ in $K_m$ of minimal period at most $m$ is hyperbolic, that is, a saddle.  Clearly, $\mathscr H_m$ is open.  For each $f\in\mathscr H_m$, the number of periodic points of $f$ in $K_m$ of minimal period at most $m$ is finite.  Call them $p_1,\ldots,p_k$. 

Choose a distance function on $X$ inducing the topology of $X$.  For each $m,j\geq 1$, let $(U_\ell)$ be a finite cover of $K_m$ by nonempty open sets of diameter less than $1/j$.  Let $\mathscr H_m^j$ be the subset of $\mathscr H_m$ of automorphisms such that if $p_1,\ldots,p_k$ are as above, then there is a transverse homoclinic point for each of $p_1,\ldots, p_k$ in $U_\ell$ for each $\ell$.  Clearly, $\mathscr H_m^j$ is open, and by Proposition \ref{t:obtain-homoclinic} applied several times, dense.  Then 
\[ \mathscr H = \bigcap_{m=1}^\infty \bigcap_{j=1}^\infty \mathscr H_m^j \]
is a dense $G_\delta$ subset of $\Aut_\omega X$ and for every $f\in \mathscr H$ and every periodic point $p$ of $f$, the transverse homoclinic points of $p$ are dense in $X$.
\end{proof}

\section{Proof of Proposition \ref{main3}} 
\label{sec:proof3}

\noindent
We consider a 2-dimensional Stein manifold $X$ with the volume density property with respect to a holomorphic volume form $\omega$.  We will use the automorphism $h:(z_1, z_2)\mapsto (z_2+z_1^d, -z_1)$ of $\C^2$, which preserves $\eta=dz_1\wedge dz_2$, with $d\geq 2$ fixed.  The higher-dimensional case may be handled in the same way, using the automorphism 
\[ (z_1, z_2, z_3,\ldots, z_n)\mapsto (z_2+z_1^d, -z_1, z_3,\ldots, z_n) \]
of $\C^n$.  It is easily seen that $h$ is H\'enon-like of degree $d$ on the bidisc $D=\{(z_1, z_2)\in\C^2:\lvert z_1\rvert, \lvert z_2\rvert <3\}$ in the sense of Dujardin \cite[Definition 2.1]{Dujardin2004}, so its entropy is $\log d$ \cite[Theorem 3.1]{Dujardin2004}.  This is the topological entropy of $h$ on the maximal completely $h$-invariant compact subset $\bigcap \limits_{k\in\mathbb Z} h^k(D)$ of $D$.  Any holomorphic injection $D\to\C^2$ sufficiently close to $h\vert_D$ is also H\'enon-like with the same degree and entropy.  It suffices to show that the open set of automorphisms $f\in \Aut_\omega X$ that are H\'enon-like of degree $d$ in some bidisc is dense.  Let $g\in\Aut_\omega X$ and let $U\subset X$ be open, relatively compact, and Runge.  We will show how to approximate $g$ on $U$ by such an automorphism $f$.

Let $W_1$ be the bidisc $s_1D$ in $\C^2$, where $s_1>0$ is small enough that $W$ admits a volume-preserving biholomorphism $\psi$, obtained by a holomorphic version of Moser's trick, onto a domain $V_1$ disjoint from $U$ and from $g(U)$, such that $U\cup V_1$ is Runge.  For $0<s_2<s_1$, let $W_2=s_2D$ and $V_2=\psi(W_2)$.  For $s_2$ small enough, there is a $C^1$ isotopy of holomorphic injections $\Psi_t: V_2\to X$, $t\in [0,1]$, preserving $\omega$, such that 
\begin{itemize}
\item $\Psi_0=g^{-1}|_{V_2}$,
\item $U$ and $\Psi_t(V_2)$ are disjoint and their union is Runge for all $t\in [0,1]$,
\item $\Psi_1(\psi(0))=\psi(0)$,
\item $\Psi_1(V_2)\subset V_1.$
\end{itemize}

For $a>0$, consider the scaling $h_a=(\frac{1}{a}\,\id)\circ h\circ (a\,\id)\in \Aut_\eta\C^2$ of $h$, mapping $(z_1, z_2)$ to $(z_2+a^{d-1}z_1^d, -z_1)$, and note that for $a$ large enough, $W_3=\frac 1 a D$ and $h_a(W_3)=\frac 1 a h(D)$ are both contained in $W_2$.  Let $V_3=\psi(W_3)$ and $\tilde h=\psi\circ h_a\circ\psi^{-1}: V_3\to V_2$.  We claim that the map $g^{-1}\circ \tilde h$  and the inclusion $V_3\hookrightarrow X$ can be joined by a piecewise $C^1$ isotopy of holomorphic injections preserving $\omega$, such that the images of $V_3$ along the isotopy are disjoint from $U$ and their unions with $U$ are Runge.  We prove this in two steps.  First, note that 
$(\Psi_t\circ \tilde h)_{t\in [0,1]}$ is a $C^1$ isotopy of holomorphic injections preserving $\omega$, joining $g^{-1}\circ \tilde h$ to the holomorphic injection
$\tilde k=\Psi_1\circ \tilde h\colon V_3\to V_1$, in such a way  that the images of $V_3$ along the isotopy are disjoint from $U$ and their unions with $U$ are Runge.

Next, the map $\tilde k$ can be joined to the inclusion $V_3\hookrightarrow X$ as follows.  Let $k=\psi^{-1}\circ\tilde k\circ \psi: W_3\to W_1$ and note that $k$ fixes the origin.  Let $L: \C^2\to \C^2$ be the derivative of $k$ at the origin.  For $t\in (0,1]$, let $\Phi_t(z)= k(tz)/t$ and let $\Phi_0=L$.  This defines a $C^1$ isotopy of holomorphic injections from $W_3$ to $\C^2$ preserving the volume form $\eta$.  By the Schwarz lemma, $\Phi_t(W_3)\subset W_1$ for all $t\in [0,1]$.  It is easy to see that $L$ can be joined to the identity on $W_3$ by a $C^1$ isotopy of linear maps $\Phi_t: W_3\to W_1$, $t\in [1,2]$, preserving $\eta$.  Then $(\psi\circ \Phi_t\circ \psi^{-1})_{t\in [0,2]}$ is a piecewise $C^1$ isotopy of holomorphic injections preserving $\omega$ and joining $\tilde k$ to the inclusion, in such a way that the images of $V_3$ along the isotopy are disjoint from $U$ and their unions with $U$ are Runge.

Now applying the conservative Anders\'en--Lempert theorem to the isotopy that joins $g^{-1}\circ \tilde h$ to the inclusion on $V_3$ and keeps the inclusion constant on $U$, we obtain automorphisms $\phi\in\Aut_\omega X$ such that $f=g\circ\phi$ approximates $g$ arbitrarily closely on $U$ and $\tilde h$ on $V_3$, so if the approximation is close enough, $f$ is H\'enon-like of degree $d$ on $V_3$ and the proof is complete.

\section{Proof of Theorem \ref{main4}} 
\label{sec:proof4}

\noindent
The proof of the first statement is a straightforward adaptation of the proof of the analogous result for automorphisms of a Stein manifold with the density property \cite[Theorem 2]{AL2025}, which in turn was a modification of the proof of the analogous result for endomorphisms of an Oka--Stein manifold \cite[Theorem 1(a)]{AL2022}.

The analogue of the second statement for automorphisms of a Stein manifold with the density property was proved as \cite[Theorem 2]{AL2025} by modifying Buzzard's proof for $\C^n$, $n\geq 2$ \cite{Buzzard1998}.  Buzzard did not consider the conservative case, but it can be handled by a similar modification.  The following paragraphs are taken from \cite{AL2025} with minor changes.

In Buzzard's notation, let $p_1$ and $p_2$ be saddle periodic points of a volume-preserving automorphism $F$ of a Stein manifold $X$ with the volume density property and $n=\dim X\geq 2$, and let $q_0 \in W^s_F(p_1)\cap W^u_F(p_2)$ be a homoclinic or heteroclinic point.  We need suitable replacements for Buzzard's families $\Psi_k$, $k\geq 1$, of automorphisms and family $\Psi$ of diffeomorphisms of $X$, defined in \cite[page 501]{Buzzard1998}.

The total $F$-orbit of $q_0$ accumulates only on the union of the cycles of $p_1$ and $p_2$, which is a finite set.   Let $K$ be the compact closure of the orbit.  Let $U$ be a Runge neighbourhood of $K$, consisting of finitely many coordinate balls with mutually disjoint closures, such that $q_0$ is the only point of $K$ in the ball $U_0$ that contains it.  By the holomorphic version of Moser's trick, the chart on $U_0$ can be taken to be volume-preserving with respect to the given volume form on $X$ and the standard volume form on $\C^n$.  Let $U_1 = U\setminus U_0$.  Let $V_0\subset U_0$ be a smaller ball containing $q_0$ and let $B$ be a closed ball in $\C^n$, centred at the origin, so that $\overline V_0+z\subset U_0$ for all $z\in B$ with respect to the coordinates in $U_0$.  Write $V=V_0\cup U_1$.

Define $\Phi:V\times B\to X$ by the formula $(x,z)\mapsto x+z$ on $V_0\times B$ and by the formula $(x,z)\mapsto x$ on $U_1\times B$.  Extend $\Phi$ to a smooth family $\Psi:X\times B\to X$ of volume-preserving diffeomorphisms.  Use the parametric Anders\'en--Lempert theorem \cite[Theorem 1.1]{Forstneric1994}\footnote{The theorem is stated for $\C^n$, $n\geq 2$, but holds more generally for Stein manifolds with the density property.  The Runge domains in the theorem should then be taken to be Stein (as they are here), but they need not be connected.  Just like the ordinary Anders\'en--Lempert theorem, the theorem may be adapted to the conservative case.}, adapted to the conservative case, to approximate $\Phi$ locally uniformly on $V\times B$ by a sequence of smooth families $\Psi_k:X\times B\to X$ of volume-preserving automorphisms of~$X$.  

These families have the properties needed for Buzzard's proof.  The remainder of the proof consists of local arguments and transversality and genericity arguments that hold in our general conservative setting.

\section{Analogous results in the non-conservative case} 
\label{sec:analogues}

\noindent
Proposition \ref{t:obtain-homoclinic} and Theorem \ref{main2} have non-conservative analogues that strengthen \cite[Theorem 8]{AL2025} and  \cite[Corollary 7]{AL2025}, respectively.  We omit the proofs, as they are very similar to the proofs given above.  For the notation used below, see \cite{AL2025}.

\begin{proposition}   \label{p:non-conservative-1}
Let $X$ be a Stein manifold with the density property.  Let $p$ be a saddle periodic point of $f\in\Aut\,X$.  Let $U$ be an open set intersecting $\tam(f)\setminus(\rne(f)\cup\rne(f^{-1}))$.  Then every neighbourhood of $f$ in $\Aut\, X$ contains an automorphism with $p$ as a saddle periodic point with a transverse homoclinic point in $U$.
\end{proposition}

\begin{theorem}  \label{t:non-conservative-2}
Let $X$ be a Stein manifold with the density property.  For a generic automorphism $f$ of $X$, if $p$ is a saddle periodic point of $f$, then the transverse homoclinic points of $p$ are dense in $J_f^*$.
\end{theorem}

In the same way that Proposition \ref{p:non-conservative-1} strengthens  \cite[Theorem 8]{AL2025}, we can improve on \cite[Theorem 4]{AL2025} as below.  The corollary then follows by an argument similar to the proof of Theorem \ref{main2}.  Again, we omit the proofs.

\begin{theorem}   \label{t:non-conservative-3}
Let $X$ be a Stein manifold with the density property.  Let $p$ be a saddle periodic point of $f\in\Aut\,X$.  Let $U$ be an open set intersecting $X\setminus\rne(f)$.  Then every neighbourhood of $f$ in $\Aut\, X$ contains an automorphism $g$ with $p$ as a saddle periodic point such that $W_g^s(p)$ intersects $U$.
\end{theorem}

\begin{corollary}  \label{c:non-conservative-4}
Let $X$ be a Stein manifold with the density property.  For a generic automorphism $f$ of $X$, if $p$ is a saddle periodic point of $f$, then $W_f^s(p)$ is dense in $J_f^+$ and $W_f^u(p)$ is dense in $J_f^-$.
\end{corollary}

Finally, the proof of Proposition \ref{main3} is easily adapted to the non-conservative case to establish the following result.

\begin{proposition}   \label{p:non-conservative-4}
A generic automorphism of a Stein manifold with the density property has infinite entropy.
\end{proposition}

\end{document}